\documentclass[10pt, twoside]{amsart}

\usepackage[utf8]{inputenc}
\usepackage{hyperref}

\usepackage[pagewise]{lineno}

\usepackage[all,cmtip]{xy}
\usepackage{tikz-cd}
\usepackage{amscd}

\usepackage{xpatch}
\makeatletter   
\xpatchcmd{\@tocline}
{\hfil\hbox to\@pnumwidth{\@tocpagenum{#7}}\par}
{\ifnum#1<0\hfill\else\dotfill\fi\hbox to\@pnumwidth{\@tocpagenum{#7}}\par}
{}{}
\makeatother    

\makeatletter
 \def\l@subsection{\@tocline{2}{0pt}{4pc}{6pc}{}}
\def\l@subsubsection{\@tocline{3}{0pt}{8pc}{8pc}{}}
 \makeatother

\usepackage{graphicx}              
\usepackage{amsmath, bm}               
\usepackage{amsfonts}              
\usepackage{amsthm}                
\usepackage{amssymb}
\usepackage{mathtools}             
\usepackage{stackengine}

\usetikzlibrary{matrix}
\usepackage{multirow}
\usetikzlibrary{shapes}

\usepackage{enumitem} 

\usepackage[all]{xy}
\usepackage{amscd}

\allowdisplaybreaks
\usepackage{color}
\usepackage{fp}

\newcommand{\Z}{\mathbb{Z}}

\newcommand{\C}{\mathbb{C}}
\newcommand{\R}{\mathbb{R}}

\newcommand{\Hi}{\mathcal{H}}
\newcommand{\Li}{\mathcal{L}}
\newcommand{\X}{\mathsf{X}}
\newcommand{\Y}{\mathsf{Y}}

\newcommand{\V}{\mathsf{V}}
\newcommand{\cj}{\overline}
\newcommand{\op}{\mathrm}
\newcommand{\B}{\op{Ba}}

\DeclarePairedDelimiterX{\normb}[1]{\lVert}{\rVert}{#1}

\numberwithin{equation}{section}
\newtheorem{theorem}{Theorem}[section]
\newtheorem{lemma}[theorem]{Lemma}
\newtheorem{proposition}[theorem]{Proposition}
\newtheorem{corollary}[theorem]{Corollary}
\renewcommand{\qedsymbol}{$\blacksquare$}
\let\origproofname\proofname
\renewcommand{\proofname}{\upshape\textbf{\origproofname}}

\newcommand{\vf}{\mathbf}

\newtheorem*{theorem*}{Theorem}
\newtheorem*{claim}{Claim}

\theoremstyle{definition}
\newtheorem{example}[theorem]{Example}
\newtheorem{definition}[theorem]{Definition}
\newtheorem{remark}[theorem]{Remark}
\newtheorem{notation}[theorem]{Notation}

\newenvironment{claimproof}[1]{\par\noindent\textbf{Proof of Claim.}\space#1}{\hfill\qedsymbol}

\usepackage{indentfirst}

\begin{document}

\title{C*-like modules and matrix $p$-operator norms}

\author{Alessandra Calin}

\author{Ian Cartwright}

\author{Luke Coffman}

\author{Alonso Delfín}

\author{Charles Girard}

\author{Jack Goldrick}

\author{Anoushka Nerella}

\author{Wilson Wu} 

\email[]{alca6792@colorado.edu, iaca9177@colorado.edu, lukecoffman@fas.harvard.edu, alonso.delfin@colorado.edu, chgi9257@colorado.edu, 
jago5326@colorado.edu, sane5566@colorado.edu, wiwu2390@colorado.edu}
\address{Department of Mathematics, University of Colorado, Boulder CO 80309, USA}

\renewcommand{\shortauthors}{Calin, 
Cartwright, Coffman, Delfín, Girard, Goldrick, Nerella, and Wu}

\date{\today}



\begin{abstract}
We present a generalization of Hölder duality to algebra-valued pairings via $L^p$-modules. 
Hölder duality states that if $p \in (1, \infty)$ and $p’$ are conjugate exponents, then the dual space of $L^p(\mu)$ is isometrically isomorphic to $L^{p’}(\mu)$. In this work, we study certain pairs $(\Y,\X)$, as generalizations of the pair $(L^{p'}(\mu), L^p(\mu))$, that have an $L^p$-operator algebra valued pairing $\Y \times \X \to A$. When the $A$-valued version of Hölder duality still holds, we say that $(\Y, \X)$ is C*-like. We show that finite and countable direct sums of the C*-like module $(A,A)$ are still C*-like when $A$ is any block diagonal subalgebra of $d \times d$ matrices. We provide counterexamples when $A \subset M_d^p(\C)$ is not block diagonal. 
\end{abstract}

\maketitle

\tableofcontents

\section{Introduction}

Hilbert C*-modules (see \cite{Lance95} for a general introduction) have been widely used 
as an operator algebraic tool. 
These objects enjoy a great deal of similarities with Hilbert spaces and therefore Hilbert C*-modules are often thought of as a generalization of Hilbert spaces. On the other hand, Hilbert spaces can
also be naturally generalized by $L^p$-spaces. 

In the past 15 years, a growing number of authors have worked on generalizations 
of C*-algebraic constructions to algebras 
of operators acting on $L^p$-spaces. This is an area now known as ``$L^p$-operator algebras''. For instance, there are $p$-analogs of: AF algebras in \cite{GLUP16,PhillipsViola20}, Cuntz algebras in \cite{DawsHor22,ncp2012AC}, graph algebras in \cite{CortRod,gardeTin25}, group, groupoid, and crossed product algebras in \cite{Bardadyn_2024,ChoiGardThiel, GLUP17,GT15,GT19,HetOrt23,NCP13}, Morita equivalence in \cite{Chung2024}, multiplier algebras in \cite{BliDelWe24}, operator spaces in \cite{Daws10,LM96}, and spectral triples in \cite{DFP24}. 

In this paper, we work with natural generalizations of concrete Hilbert C*-modules to the $L^p$-setting. These objects 
are called $L^p$-modules and roughly can be defined as a pair that, on one side generalizes the fact that every Hilbert space is a right Hilbert $\C$-module, and on the other side generalizes that every Hilbert space is an $L^p$-space. To be more precise, let $p, p'$ be Hölder conjugate exponents, that is let $p, p' \in [1, \infty]$ satisfy
\[
\frac{1}p + \frac{1}{p'} =1.
\]
Then, any Hilbert space $\Hi$ can be thought as a pair $(L^{p'}(\Omega,\mu), L^p(\Omega,\mu))$ by taking $p=2$ and $(\Omega, \mu)$ an appropriate measure space. At the same time, $\Hi$ also gives rise to a pair of Hilbert $\C$-modules $(\widetilde{\V},\V)$ where $\widetilde{\V}$ is the conjugate vector space of the module $ \V$, indeed just take $\V$ to be simply $\Hi$. Then, 
our definition of an $L^p$-module over an $L^p$-operator algebra $A$ will yield a pair $(\Y, \X)$, together with an $A$-valued pairing $(-\mid-)_A \colon \Y \times \X \to A$, such that 
the following diagram of ``generalizations'' commutes:
\[\begin{tikzcd}        
    \Hi \arrow[hook]{dr} \arrow[hook]{rr} & & (\widetilde{\V},\V) \arrow[hook]{dr}  &  \\
    & (L^{p'}(\Omega,\mu), L^{p}(\Omega, \mu)) \arrow[hook]{rr} & & (\Y, \X).
\end{tikzcd}
\]
See Definition \ref{DefLpMod} for the complete details. 

At the Hilbert C*-module level, it is well known that the module norm is always recovered, as an operator norm, by the C*-valued inner product (see Lemma \ref{C*innerRec}). This also happens for pairs $(L^{p'}(\mu), L^p(\mu))$ thanks to Hölder duality (see Theorem \ref{HoldDual}). However, this is generally no longer true for $L^p$-modules. Those pairs $(\Y, \X)$ for which the module norms, that is the $\Y$ and $\X$ norms, are recovered by the pairing are called \textit{C*-like $L^p$-modules} (see Conditions \eqref{C*likeness1} and \eqref{C*likeness2} in Definition \ref{DefLpMod}). The main goal of this work is to analyze, in some instances, the condition of being C*-like. 

Our main result is Theorem \ref{MainTT}, in which we prove that if $A$ is any block diagonal subalgebra of $d \times d$ matrices equipped with the $p$-operator norm for $p \in [1,\infty)$, then, for any $n \in \Z_{\geq 1} \cup \{\infty\}$, the row-column module $(M^p_{1,n}(A), M_{n,1}^p(A))$ (see Examples \ref{colrow} and \ref{INFcolrow}) is a C*-like module over $A$. Part of the difficulties of establishing results in this direction is that $p$-operator matrix norms, unless $p\in\{1,2,\infty\}$, are NP-hard to compute (see \cite{HendOlshec10}). As a consequence, in Corollary \ref{middleSteps}, we establish interesting and useful results about matrix $p$-operator norms of horizontally and vertically stacked matrices. We also present useful properties of $p$-operator norms for block diagonal matrices in Lemma \ref{KeyBlockNorms}.

The paper is organized as follows. Section \ref{PreLim} establishes the notation and contains a short overview of known results from Hilbert C*-modules and $L^p$-operator algebras that will be used throughout the paper. Section \ref{LpandC*like} deals with the definitions of $L^p$-modules, the C*-likeness conditions, and presents some examples. In Section \ref{Afdim} we focus on $L^p$-modules over matrix algebras and present our main results. Finally, in Section \ref{CounterExs} we present two counterexamples of matrix algebras whose row-column module is not C*-like. 

\section*{Acknowledgments} 
Part of this work was carried during an REU(G) project funded by the University of Colorado, Boulder in the summer of 2024. The authors are grateful for the support provided. In particular, we would like to thank Nat Thiem for the excellent organizational work and also for suggesting us to look at a simultaneously diagonalizable case, as we ended up doing in Example \ref{SD}. 

\section{Preliminaries and notation}\label{PreLim}

\subsection{Banach spaces and operator algebras}

We remind the reader that a \textit{Banach space} is a 
complex vector space $E$ equipped with a norm $\|-\|_E$ that makes $E$ 
a complete metric space. 
If $E$ is a Banach space, we will denote its closed unit ball by 
$\B(E)$, that is
\[
\B(E) = \{ \xi \in E \colon \| \xi\|_E \leq 1 \}.
\]

A Banach space $A$ that is also a complex algebra is called a \textit{Banach algebra} if the norm is submultiplicative, that is $\| a b \|_A \leq \|a\|_A\|b\|_A$ for all $a,b \in A$. 
If $E_0$ and $E_1$ are Banach spaces, we write $\Li(E_0,E_1)$
for the space of bounded linear maps from $E_0$ to $E_1$, 
which comes equipped with the operator norm, that is for $a \in \Li(E_0,E_1)$, 
\[
\|a\| = \sup \{ \| a \xi \|_{E_1} \colon { \xi \in \B(E_0)}\}.
\]
In particular, we write $\Li(E)=\Li(E,E)$, 
which is a Banach algebra. 
The topological dual of a Banach space $E$ will be denoted by $E'$, that is 
$E'=\Li(E, \C)$. 

A \textit{Hilbert space} $\Hi$ is a Banach space whose norm is induced by a 
complex valued inner product on $\Hi$ via
\[
\|\xi \|_\Hi = \langle \xi, \xi \rangle ^{1/2}.
\]
If $\Hi_0$ and $\Hi_1$ are Hilbert spaces, any 
$a \in \Li(\Hi_0, \Hi_1)$ is characterized by the existence of its adjoint
$a^* \in \Li(\Hi_1, \Hi_0)$ 
satisfying 
\[
\langle a\xi, \eta \rangle = \langle \xi, a^*\eta \rangle
\]
for any $\xi \in \Hi_0$, $\eta \in \Hi_1$. 
\begin{definition}
A \textit{concrete C*-algebra} $A$ is a closed and selfadjoint subalgebra 
of $\Li(\Hi)$ for a Hilbert space $\Hi$. An \textit{abstract C*-algebra}
is a Banach algebra $A$ with an involution $a\mapsto a^*$ such that the 
C*-equation holds: that is $\| a^*a\|=\|a\|^2$ for all $a \in A$. 
\end{definition}

\begin{remark}
It is well known, and not hard to check, that any concrete C*-algebra is also an abstract one.
The celebrated GNS construction, see for instance Theorem 3.4.1 in \cite{Murphy90}, shows that any abstract C*-algebra 
is $*$-isomorphic to a concrete C*-algebra. That is, 
if $A$ is any abstract C*-algebra, then there is a Hilbert space $\Hi$ 
and an injective $*$-preserving algebra homomorphism 
$\varphi \colon A \to \Li(\Hi)$. Furthermore, an important result in C*-algebras (see Theorem 3.1.5 in \cite{Murphy90}) says that since 
$\varphi$ is injective, then $\varphi$ is also isometric. 
Thus, we can and will identify 
$A$ with its isometric copy $\varphi(A)$ in 
$\Li(\Hi)$. 
\end{remark}
If $(\Omega, \mathfrak{M}, \mu)$ is a measure space and $E$ is a Banach space, two $E$-valued measurable functions are equivalent when their difference vanishes almost everywhere with respect to $\mu$. Then we define $L^0(\Omega \to E)$ to be the vector space of equivalence classes of $E$-valued measurable functions. For $p \in [1, \infty)$, we denote  by 
$L^p(\Omega, \mathfrak{M}, \mu)$ 
the Banach space of $p$-integrable complex valued functions equipped 
with the usual $p$-norm. That is, 
\[
 L^p(\Omega, \mathfrak{M}, \mu) = \left\{
f \in L^0(\Omega \to \C) \colon \int_\Omega |f|^p d\mu < \infty
 \right\}
\]
and 
\[
\|f \|_p =  \| f \|_{L^p(\Omega, \mathfrak{M}, \mu)} = 
\left( \int_\Omega |f|^p d\mu \right)^{1/p}.
\]
If $p=\infty$, we have 
\[
 L^p(\Omega, \mathfrak{M}, \mu) = \left\{
f \in L^0(\Omega \to \C) \colon \| f \|_\infty < \infty
 \right\}
\]
where
\[
\| f \|_\infty = \| f \|_{L^\infty(\Omega, \mathfrak{M}, \mu)}=
\inf \{ \alpha > 0 \colon \mu(\{w \in \Omega \colon |f(w)|>\alpha\})=0\}.
\]
For $p \in [1, \infty]$, we will often
write $L^p(\mu)$ instead of $L^p(\Omega, \mathfrak{M}, \mu)$. 
When $\nu$ is the counting measure on a set $\Omega$, 
we simply write $\ell^p(\Omega)$ to refer to $L^p(\Omega, 2^\Omega, \nu)$. 
In particular, when $d \in \Z_{\geq 1}$
we write $\ell_d^p$ instead of $\ell^p(\{1, \ldots, d\})$. 
\begin{definition}
For $p \in [1, \infty]$, an \textit{$L^p$-operator algebra} is 
a Banach algebra $A$ such that there is a measure space 
$(\Omega, \mathfrak{M}, \mu)$ and an isometric 
algebra homomorphism $\varphi \colon A \to \Li(L^p(\mu))$. 
\end{definition}
 Note that any C*-algebra is an $L^2$-operator 
algebra, but the converse is not true as $L^2$-operator 
algebras can fail to be selfadjoint. There is, currently, not a known abstract characterization of $L^p$-operator algebras; therefore, we always work on the concrete setting by specifying a representation on an $L^p$-space. 

\begin{notation}\label{MatrixCOnv}
We will be mostly working with spaces of matrices equipped 
with $p$-operator norms. Specifically, for $d, n \in \Z_{\geq 1}$ 
and $p \in [1, \infty]$, we fix the following notation
\[
M_{n,d}^p(\C) = \Li(\ell_d^p, \ell_n^p), \ 
M_d^p(\C) =M_{d,d}^p(\C) =  \Li(\ell_d^p). 
\]
\end{notation}
\subsection{Hölder duality} 

For reference, we briefly recall basic well known statements about Hölder duality 
that will be used throughout this paper. For each $p \in [1, \infty]$ we 
let $p' \in [1,\infty]$ be its Hölder conjugate, that is 
$p'=\frac{p}{p-1}$ so that $\frac{1}{p}+\frac{1}{p'}=1$. 
For a measure space $(\Omega, \mathfrak{M}, \mu)$
we have the dual pairing 
$\langle -, - \rangle_p \colon L^{p'}(\mu) \times L^{p}(\mu) \to \C$
given by
\begin{equation}\label{LqLpPairing}
  \langle \eta, \xi \rangle_p =  \int_\Omega \xi\eta  d\mu.  
\end{equation}
We remark that if $\Hi=\ell^2(J)$ is any Hilbert space, 
then the inner product on $\Hi$ is recovered by 
$\langle \xi, \eta \rangle = \langle  \eta, \cj{\xi} \rangle_2$.
The following result is folklore:

\begin{theorem}\label{HoldDual}
Let $p \in [1,\infty]$ and let $(\Omega, \mathfrak{M}, \mu)$ be a measure space. 
Define  $\Phi \colon L^{p'}(\mu) \to L^{p}(\mu)'$ by 
$\Phi(\eta)(\xi)=\langle \eta, \xi\rangle_p$,
for $\eta \in  L^{p'}(\mu)$ and $\xi \in  L^{p}(\mu)$. 
\begin{enumerate}
    \item If $p \in (1,\infty)$, then $\Phi$ is an isometric isomorphism 
    from $L^{p'}(\mu)$ to $L^{p}(\mu)'$
   \item  If $p=1$ and $\mu$ is sigma-finite, then $\Phi$ 
   is an isometric isomorphism from $L^\infty(\mu)$ to $L^1(\mu)'$. 
   \item If $p=\infty$ and $\mu$ is the counting measure on $\{1, \ldots, d\}$ 
   for $d \in \Z_{\geq 1}$, then $\Phi$ 
   is an isometric isomorphism from $\ell^1_d$ to $(\ell^\infty_d)'$.
\end{enumerate}
\end{theorem}
\begin{proof}
Parts (1) and (2) follow from Theorem 6.15 in \cite{FollGB99}.
 For Part (3), since we are working in finite dimension, 
 a direct check shows that $\Phi \colon \ell^1_d \to (\ell_d^\infty)'$ is indeed an isometric isomorphism. 
\end{proof}

An immediate consequence of the previous result, which we will often use, is that 
the norm of a matrix as an operator in $\Li(\ell_d^p, \ell_n^q)$
can be obtained by solving a double maximization problem:
\begin{corollary}\label{FiniteDopNorm}
Let $d, n \in \Z_{\geq 1}$, let $p, q \in [1, \infty]$, and let 
$a \in \Li(\ell_d^p, \ell_n^q)$. 
Then
\[
\|a \|_{ \Li(\ell_d^p, \ell_n^q)} = \max \left\{ |\langle \eta, a\xi \rangle_q|
\colon 
\xi \in \B(\ell_d^p), \eta \in \B\big(\ell_n^{q'}\big) \right\}.
\]
\end{corollary}

From the last corollary, we also obtain a particular case of the well known result that the norm of an operator agrees with the one of its dual (see Proposition VI.1.4 in \cite{ConFA90} for instance). In our scenario, this means that the $p \to q$ operator norm of any matrix agrees with the $q' \to p'$ operator norm of its transpose:

\begin{corollary}\label{a_p=a^T_q}
    Let $d, n \in \Z_{\geq 1}$, let $p, q \in [1, \infty]$, and let 
$a \in \Li(\ell_d^p, \ell_n^q)$. Then $a^T \in \Li(\ell_n^{q'}, \ell_d^{p'})$ and 
\[
\| a \|_{ \Li(\ell_d^p, \ell_n^q)} = \| a^T \|_{\Li(\ell_n^{q'}, \ell_d^{p'})}
\]
\end{corollary}
\begin{proof}
It follows at once from Corollary \ref{FiniteDopNorm} after observing that 
    \[
    \langle \eta, a\xi \rangle_{q} = \langle \xi , a^T \eta \rangle_{p'}
    \]
    for any $\eta \in \ell_n^{q'}$ and $\xi \in \ell_d^p$.
\end{proof}

\subsection{Spatial tensor product}\label{pTenProd}

For $p \in [1, \infty)$, there is a tensor product, called the 
\textit{spatial tensor product} and denoted by $\otimes_p$. 
We refer the reader to Section 7 of \cite{defflor1993} for complete 
details on this 
tensor product. We only describe below the properties we will need. 
If $(\Omega_0, \mathfrak{M}_0, \mu_0)$ is a measure space and $E$ is a 
Banach space, 
then there is an isometric isomorphism 
\[
L^p(\mu_0) \otimes_p E \cong L^p(\mu_0, E)=
\{ g \in  L^0(\Omega_0 \to E)
\colon 
\textstyle{\int_{\Omega_0}} \|g(x)\|_E^p \  d\mu_0(x) < \infty\}
\]
such that for any $\xi \in L^p(\mu_0)$ and $\eta \in E$, 
the elementary tensor $\xi \otimes \eta $ 
is sent to the function $\omega \mapsto \xi(\omega)\eta$. 
Furthermore, if $(\Omega_1,  \mathfrak{M}_1, \mu_1)$
is another measure space and $E=L^p(\mu_1)$, then there is 
an isometric isomorphism
\begin{equation}\label{LpTensor}
L^p(\mu_0) \otimes_p L^p(\mu_1) \cong 
L^{p}(\Omega_0 \times \Omega_1, \mu_0 \times \mu_1),
\end{equation}
sending $\xi \otimes \eta$ to the function 
$(\omega_0,\omega_1) \mapsto \xi(\omega_0)\eta(\omega_1)$ 
for every $\xi \in L^p(\mu_0)$
and $\eta \in L^p(\mu_1)$. 
We describe its main properties below. 
The following is Theorem 2.16 in \cite{ncp2012AC}, 
except that we have removed the $\sigma$-finiteness assumption 
as in the proof in Theorem 1.1 in \cite{figiel1984}. 
\begin{theorem}\label{SpatialTP}
Let $p \in [1, \infty)$ and for $j \in \{0,1\}$ let 
$(\Omega_j, \mathfrak{M}_j, \mu_j)$, 
$(\Lambda_j, \mathfrak{N}_j, \nu_j)$ be measure spaces.
\begin{enumerate}
\item Under the identification in \eqref{LpTensor}, 
$\op{span}\{ \xi \otimes \eta  \colon \xi \in L^p(\mu_0), \eta \in L^p(\mu_1)\}$ 
is a dense subset of 
$L^{p}(\Omega_0 \times \Omega_1, \mu_0 \times \mu_1)$.  \label{denseTensor_p}
\item $\| \xi \otimes \eta \|_p = \| \xi\|_p\|\eta\|_p$  for every 
$\xi \in L^p(\mu_0)$ and $\eta \in L^p(\mu_1)$. \label{LpTensorNorm}
\item Suppose that 
$a \in \Li(L^p(\mu_0), L^p(\nu_0))$
and $b \in \Li(L^p(\mu_1), L^p(\nu_1))$. 
Then there is a unique map 
$a\otimes b \in \Li(L^p(\mu_0 \times \mu_1),
 L^p( \nu_0 \times \nu_1))$ 
such that 
\[
(a \otimes b)(\xi \otimes \eta)=a\xi \otimes b\eta
\]
for every $\xi \in L^p(\mu_0)$ and $\eta \in L^p(\mu_1)$. 
Further, $\| a\otimes b\|=\|a\|\|b\|$. \label{LpTensorOp}
\item The tensor product of operators defined in 
\eqref{LpTensorOp} is associative, bilinear, 
and satisfies (when the domains are appropriate) 
$(a_1 \otimes b_1)(a_2 \otimes b_2) = a_1 a_2 \otimes b_1b_2$. \label{LpTensorOpProp}
\end{enumerate}
\end{theorem}

\subsection{Hilbert C*-modules}\label{Hc*RECAP}

Our main objects of study in this paper will be certain modules over 
$L^p$-operator algebras that generalize the notion of (right)
Hilbert modules over C*-algebras. In this subsection we 
recall some of the basic theory of Hilbert modules and prove 
facts that will be used for comparison with the $L^p$-case. 

\begin{definition}
    Let $A$ be a C*-algebra. A \textit{(right) Hilbert $A$-module}
is a complex vector space $\X$ that is a (right) $A$-module 
    equipped 
    with a (right) $A$-valued inner product 
    $\langle - , - \rangle_A \colon \X \times \X \to A$
    such that $\X$ is complete with the norm 
    \[
    \|x\|= \|\langle x, x \rangle_A\|^{1/2}.
    \]
\end{definition}

For any right Hilbert $A$-module $\X$ we let 
$\widetilde{\X}$ be its conjugate vector space. That is, 
as an additive group $\widetilde{\X}$ is isomorphic to $\X$ (that is we have an additive group isomorphism $x \mapsto \tilde{x} \in \widetilde{\X}$) and as a vector space $\widetilde{\X}=\{\tilde{x} \colon x \in \X\}$ has scalar 
multiplication given by
\[
\lambda \cdot \tilde{x} = \widetilde{ \cj{\lambda} x},
\]
for any $\lambda \in \C$, $x \in \X$. 
Note that $\widetilde{\X}$ has a natural 
structure of a left Hilbert $A$-module with left action 
and left $A$-valued inner product given by 
\[
a \cdot \tilde{x} = \widetilde{xa^*}, 
\ {}_A\langle  \tilde{x}, \tilde{y} \rangle = \langle x,y \rangle_{A}.
\]
Therefore, the $A$-valued inner product is also interpreted 
as a pairing between $\widetilde{\X}$ and $\X$ by letting 
$( - \mid - )_{A} \colon \widetilde{\X} \times \X \to A$ be defined 
as 
\begin{equation}\label{DualC*pair}
( \tilde{x} \mid y )_{A} =  \langle x,y \rangle_{A} = {}_A\langle  \tilde{x}, \tilde{y} \rangle.
\end{equation}
Just as with C*-algebras, there is a concrete definition for Hilbert modules. 
\begin{definition}
Let $A \subseteq \Li(\Hi_0)$ be a concrete C*-algebra acting on a Hilbert space $\Hi_0$ and let 
$\Hi_1$ also be a Hilbert space. 
A \textit{concrete right Hilbert $A$-module} is a 
closed subspace $\X \subseteq \Li(\Hi_0, \Hi_1)$ 
such that $xa \in \X$ for all $x \in \X$, $a \in A$ 
and $x^*y \in A$  for all $x,y \in \X$. 
\end{definition}
It is clear that any concrete Hilbert $A$-module 
is also a Hilbert $A$-module with right action and 
$A$-valued inner product defined by
\[
x\cdot a = xa \in \X, \ \langle x, y \rangle_A = x^*y \in A,
\]
see Proposition 2.2 in \cite{Delfin24} for the details. 
Conversely, it is also known
that any right Hilbert $A$-module $\X$ can be isometrically represented as a concrete one:
\begin{theorem}\label{ConcreteMod}
    Let $A$ be a C*-algebra,
    let $\X$ be any right Hilbert $A$-module, and let 
    $\varphi_A \colon A \to \Li(\Hi_0)$ be any faithful nondegenerate representation of $A$ on a Hilbert space $\Hi_0$. Then there is a Hilbert space $\Hi_1$ and an isometric linear map $\pi_\X \colon \X \to \Li(\Hi_0, \Hi_1)$
    such that $\pi_\X(\X)$ is a concrete right Hilbert $\varphi_A(A)$-module. Furthermore, $\pi_\X(xa)=\pi_\X(x)\varphi_A(a)$ and $\varphi_A(\langle x,y\rangle_A)=\pi_\X(x)^*\pi_\X(y)$ for all $x,y \in \X$ and $a \in A$. 
\end{theorem}
\begin{proof}
    This is well known and has been shown using different methods. See, for instance, Theorem 2.6 in \cite{Zettl83}, Theorem 3.1 in \cite{Murphy97}, and Corollary 3.6 in \cite{Delfin24}.
\end{proof}

Now assuming that $\X \subseteq \Li(\Hi_0, \Hi_1)$ is a concrete right Hilbert $A$-module, then $\X^* = \{ x^* \colon x \in \X\} \subseteq \Li(\Hi_1, \Hi_0)$ is  
a concrete left Hilbert $A$-module with action $a \cdot x^*=(xa^*)^*$ and left $A$-valued inner
product ${}_A\langle x^*, y^* \rangle = x^*y$. 
It is clear that 
$\widetilde{\X}$ and $\X^*$ are isomorphic 
as left Hilbert modules via the map $\tilde{x} \mapsto x^*$. Thus we have shown the following:
\begin{proposition}\label{ConcreteConj}
     Let $A$ be a C*-algebra,
    let $\X$ be any right Hilbert $A$-module, let 
    $\varphi_A \colon A \to \Li(\Hi_0)$ be any faithful nondegenerate representation of $A$ on a Hilbert space $\Hi_0$, and let $\pi_\X$ be the map from Theorem \ref{ConcreteMod}. Then $\widetilde{\X}$ and $\pi_\X(\X)^*$ are isomorphic
as left Hilbert modules.
\end{proposition}

We now state, for reference,
a well known result
for computing norms in Hilbert $A$-modules:

\begin{lemma}\label{C*innerRec}
    Let $A$ be a C*-algebra and let $\X$ be a right Hilbert
    module. Then, for any $x \in \X$,
    \[
    \| x \| = \sup \{ \|\langle x,y \rangle_A\| \colon 
    y \in \B(\X) \}.
    \]
\end{lemma}
\begin{proof}
See Lemma 2.18 and (2.14) in \cite{RaeWil98}. 
\end{proof}

Observe that the previous result says that the
pairing \eqref{DualC*pair} recovers
the norms in both $\X$ and $\widetilde{\X}$. Indeed, 
since the map $x \mapsto \tilde{x}$ is an isometry, 
the previous lemma gives at once, for any $x \in \X$, 
that 
\begin{equation}\label{DualC*pairR1}
\| x \|= \sup \left\{  \| ( \tilde{y} \mid x )_{A} \|
\colon  \tilde{y} \in \B\big(\widetilde{\X}\big)
\right\}
\end{equation}
and 
\begin{equation}\label{DualC*pairR2}
\|\tilde{x}\| = \sup \left\{ \| ( \tilde{x} \mid y )_{A} \|
\colon  y \in \B\big(\X\big)
\right\}
\end{equation}
One of the main goals of this paper is to study certain modules over $L^p$-operator algebras that have an analog of the pairing \eqref{DualC*pair}. When $p=2$, these modules include Hilbert C*-modules. We will see however, that in general it is not always true that Equations \eqref{DualC*pairR1} and \eqref{DualC*pairR2} 
hold, and that it is not always easy to decide when they do. Indeed, this could be the case even for $p=2$ (in the non-selfadjoint case) as it is shown in the second part of Example \ref{AA}, where we provide instances for which the $p$-analogs of Equations \eqref{DualC*pairR1} and \eqref{DualC*pairR2} fail for every $p \in [1, \infty)$. 

\section{$L^p$-modules and C*-likeness}\label{LpandC*like}

We begin by recalling the general definition of an $L^p$-analog of Hilbert modules. The following
is Definition 3.1 in \cite{Delfin24lpmodules}, which was originally motivated
by Theorem \ref{ConcreteMod} 
above. 

\begin{definition}\label{DefLpMod}
    Let $(\Omega_0, \mathfrak{M}_0, \mu_0)$ and 
    $(\Omega_1, \mathfrak{M}_1, \mu_1)$ be measure spaces, 
    let $p \in [1, \infty]$, 
    and let $A \subseteq \Li(L^p(\mu_0))$ be an $L^p$-operator algebra. 
    An $L^p$-module over $A$ is a pair $(\Y, \X)$ such that 
    $\X \subseteq \Li(L^p(\mu_0), L^p(\mu_1))$ and $\Y \subseteq \Li(L^p(\mu_1), L^p(\mu_0))$ 
    are closed subspaces satisfying 
    \begin{enumerate}
        \item $xa \in \X$ for all $x \in \X$, $a \in A$,  
        \item $yx \in A$ for all $x \in \X$, $y \in \Y$, \label{Lp2}
        \item $ay \in \Y$ for all $y \in \Y$, $a \in A$. 
    \end{enumerate}
Condition \eqref{Lp2} gives a pairing $(-\mid -)_A \colon \Y \times \X \to A$ defined by $(y \mid x)_A = yx$. 
We say that the module $(\Y, \X)$ is \textit{C*-like}
if both the norms in $\X$ and $\Y$ are recovered by the 
pairing, that is for every $x \in \X$ and 
$y \in \Y$ we have 
\begin{enumerate}
 \setcounter{enumi}{3}
    \item $\|x\| = \sup \{ \| (y \mid x)_A \| 
    \colon y \in \B(\Y) \}$, \label{C*likeness1}
    \item $\|y\| = \sup\{ \| (y \mid x)_A \|
    \colon x \in \B(\X)\}$. \label{C*likeness2}
\end{enumerate}
    
\end{definition}

\begin{remark}
Let $(\Y, \X)$ be an $L^p$-module over $A$. 
Then $(\Y, \X)$ is a Banach $A$-pair in the sense of \cite{Laf02, Para09}.
\end{remark}

\begin{remark}
        Let $(\Y, \X)$ be an $L^p$-module over $A$. Then both $\X$ and $\Y$ are naturally concrete $p$-operator spaces 
as defined in Section 4.1 of \cite{Daws10}.
\end{remark}

\begin{remark}\label{DefinitionLpModRmks}
For an $L^p$-module $(\Y, \X)$  over $A$,
observe that the pairing $\Y \times \X \to A$ satisfies 
\[
\|(y \mid x )_A \| \leq \| y \|\| x\|,
\]
for all $x \in \X$, $y \in \Y$. Hence, the C*-likeness conditions \eqref{C*likeness1} and 
\eqref{C*likeness2}
in Definition \ref{DefLpMod}
can be restated as 
\begin{enumerate}
 \setcounter{enumi}{3}
    \item $\|x\| \leq \sup \{ \| (y \mid x)_A \| 
    \colon y \in \B(\Y) \}$,
    \item $\|y\| \leq  \sup\{ \| (y \mid x)_A \|
    \colon x \in \B(\X)\}$.
\end{enumerate}
\end{remark}

Next we point out that the notion of (C*-like) $L^p$-modules 
does indeed generalize
right Hilbert C*-modules.

\begin{proposition}
    Let $A$ be a C*-algebra and let $\X$ be any right Hilbert $A$-module, then $( \widetilde{\X},\X)$ is a C*-like $L^2$-module over $A$.
\end{proposition}
\begin{proof}
    This follows at once from combining Equations \eqref{DualC*pairR1} and \eqref{DualC*pairR2} above together with Theorem \ref{ConcreteMod} and Proposition \ref{ConcreteConj}. See also Example 3.4 in \cite{Delfin24lpmodules}. 
\end{proof} 

\begin{example}\label{AA}
Fix $p \in [1, \infty)$. Let $A$ be any $L^p$-operator algebra that has
a contractive approximate unit. That is, there is a net $(e_\lambda)_{\lambda \in \Lambda}$ in $A$ such that 
$\|e_\lambda\| \leq 1$ for all $\lambda \in \Lambda$ and
\[
\| e_\lambda a-a\|, \|ae_\lambda-a\| \to 0
\]
for all $a \in A$. Then, it is immediate to check that the pair $(A, A)$
is a C*-like $L^p$-module over $A$. 

Notice that the lack of a contractive approximate unit can result in the C*-like condition to fail for $(A,A)$. Indeed,   let $T_2^p \subset \Li(\ell_2^p)$ be the algebra of strictly upper triangular matrices acting on $\ell_2^p$. This algebra will be revisited in Example \ref{ex_UpperTriang} below. Observe that $T_2^p$ is an $L^p$-operator algebra and in fact is isometrically isomorphic to the complex numbers $\C$ with the trivial multiplication 
(i.e. $ab=0$ for all $a,b \in \C$). It is then clear that $(T_2^p,T_2^p)$ is an $L^p$-module over $\C$. However, $T_2^p$ does not have a contractive approximate unit and the fact that multiplication is trivial gives at once that $(T_2^p,T_2^p)$ is not C*-like (see Example 3.5 in \cite{Delfin24lpmodules} for the details). This is still valid when $p=2$; that is, the pair $(T_2^2,T_2^2)$ is an $L^2$-module over $\C$ (where $\C$ is thought of as the non-selfadjoint $L^2$-operator algebra $T_2^2$) that is not C*-like.
\end{example}

\begin{example}\label{LpSpacesAsModules}
    Let $(\Omega, \mathfrak{M}, \mu)$ be any measure space and let $p \in [1, \infty]$. Observe that $\Li(\ell_1^p, L^p(\mu))$ is isometrically isomorphic to $L^p(\mu)$ via the map $a \mapsto a(1)$. Similarly, for $p \not\in \{1, \infty\}$ Hölder duality gives that $\Li(L^p(\mu), \ell_1^p)$ is isometrically isomorphic to $L^{p'}(\mu)$ 
where $p'$ is the Hölder conjugate for $p$. This last statement is still true for $p=1$ if $\mu$ is sigma-finite, and also for $p=\infty$ when $\Omega$ is a finite set with any weighted counting measure. Thus, from Theorem \ref{HoldDual} we obtain: 
\begin{enumerate}
    \item If $p \in (1, \infty)$, then $(L^{p'}(\mu), L^{p}(\mu))$ is a C*-like $L^p$-module over $\C=\Li(\ell_1^p)$. \label{LpLq}
     \item If $p=1$ and $\mu$ is sigma-finite, then $(L^{\infty}(\mu), L^{1}(\mu))$ is a C*-like $L^1$-module over $\C=\Li(\ell_1^p)$.\label{L1Lin}
     \item  For $p \in [1, \infty]$ and $n \in \Z_{\geq 1}$ we have that $(\ell_n^{p'}, \ell_n^{p})$ is a C*-like $L^p$-module over $\C=\Li(\ell_1^p)$. \label{lplq}
\end{enumerate}
In all three cases above, the module pairing is the standard one from Equation \eqref{LqLpPairing}. That is, 
$(\eta \mid \xi )_\C = \langle \eta , \xi \rangle_p$. 
\end{example}

We can also get 
new $L^p$-modules out of old ones: 

\begin{example}[Row-column module]\label{colrow}
Fix $p \in [1, \infty)$ and $n \in \Z_{\geq 1}$, let $A \subseteq \Li(L^p(\mu_0))$ be an 
$L^p$-operator algebra, and let $(\Y, \X)$ be an 
$L^p$-module over $A$ as in Definition \ref{DefLpMod}. 
Let $M^p_{1,n}(\Y)$ the space 
of row $n$-dimensional vectors with entries in $\Y$ and let $M^p_{n,1}(\X)$ be the space of column $n$-dimensional vectors 
with entries in $\X$. 
Notice that we can equip both $M^p_{1,n}(\Y)$ and 
$M^p_{n,1}(\X)$
with natural $p$-operator space structures by realizing 
them as the module from Example \ref{LpSpacesAsModules}\eqref{lplq} tensored with $(\Y, \X)$ via the spatial tensor product described 
in Theorem \ref{SpatialTP}:
\begin{align*}
M^p_{1,n}(\Y) & = \ell_n^{p'} \otimes_p \Y \cong \Li(\ell_n^p, \ell_1^p) \otimes_p \Y \subseteq \Li(L^p(\nu_n \times \mu_1), L^p(\mu_0) )\\
M^p_{n,1}(\X) & = \ell_n^p \otimes_p \X \cong \Li(\ell_1^p, \ell_n^p) \otimes_p \X \subseteq \Li(L^p(\mu_0), L^p(\nu_n \times \mu_1) ),
\end{align*}
where $\nu_n$ is counting measure on $\{1, \ldots, n\}$. 
Then, $( M^p_{1,n}(\Y), M^p_{n,1}(\X))$ is an $L^p$-module over 
$A$. Furthermore, observe that $( M^p_{1,n}(\Y), M^p_{n,1}(\X))$ naturally corresponds to the direct sum of the module $(\Y, \X)$: 
\[
( M^p_{1,n}(\Y), M^p_{n,1}(\X)) = \bigoplus_{j=1}^n (\Y, \X)
\]
at the spatial level. We will call $( M^p_{1,n}(\Y), M^p_{n,1}(\X))$ the \textit{row-column module of $(\Y, \X)$}. 
\end{example}

\begin{remark}\label{matrixNorm}
    For any $p \in [1,\infty)$, as a particular case of Example \ref{LpSpacesAsModules}\eqref{lplq} we have that the pair $(\C, \C)=(\ell_1^{p'}, \ell_1^{p})$ is an $L^p$-module over $\C$. Observe then that the row-column module of $(\C, \C)$ (see Example \ref{colrow}) is in fact $( M_{1,n}^p(\C),M_{n,1}^p(\C))$, where we are now interpreting the pair as in Notation \ref{MatrixCOnv}. That is, the two potential interpretations for $M_{1,n}^p(\C)$ and $M_{n,1}^p(\C)$ have the same meaning. 
\end{remark} 

For convenience and future reference we explicitly state 
below the $L^p$-module structure for 
$( M^p_{1,n}(\Y),M^p_{n,1}(\X))$ from Example \ref{colrow} above. For $\vf{x} \in M^p_{n,1}(\X)$, 
$\vf{y} \in  M^p_{1,n}(\Y)$, and $a \in A$ we have 
\[
\vf{x} \cdot a = \begin{bmatrix}
    x_1 \\
    \vdots \\
    x_n
\end{bmatrix} \cdot a = 
\begin{bmatrix}
    x_1a \\
    \vdots \\
    x_na
\end{bmatrix},
\ \ \ \ 
a \cdot \vf{y}=a \cdot \begin{bmatrix}
    y_1 & \ldots & y_n
\end{bmatrix} = \begin{bmatrix}
    ay_1 & \ldots & ay_n
\end{bmatrix}, 
\]
and 
\[ (\vf{y} \mid \vf{x})_A = 
\begin{bmatrix}
    y_1 & \ldots & y_n
\end{bmatrix} \cdot \begin{bmatrix}
    x_1 \\
    \vdots \\
    x_n
\end{bmatrix} = \sum_{j=1}^n y_jx_j = \sum_{j=1}^n (y_j \mid x_j)_A.
\]
Moreover, for $\vf{x} \in M^p_{n,1}(\X)$ we have 
\[
\| \vf{x} \|= \sup \left\{ \Bigg(\sum_{j=1}^n  \left\|x_j\xi\right\|_p^p\Bigg)^{1/p} \colon 
\xi \in \B(L^p(\mu_0))
\right\},
\]
and for $\vf{y} \in  M^p_{1,n}(\Y)$ 
\[
\| \vf{y} \|= \sup \left\{  \Bigg\| \sum_{j=1}^n y_j\eta_j\Bigg\|_p \colon 
(\eta_1, \ldots, \eta_n) \in \B\big( L^p(\nu_n \times \mu_1) \big)  
\right\}.
\]

One of the main goals of this paper is to study the C*-likeness for certain $L^p$-modules of the form $( M^p_{1,n}(\Y),M^p_{n,1}(\X))$. This problem relies on norm computations, and even when both $\X$ and $\Y$ are finite dimensional modules, all the norms boil down to $p$-operator norms which are known to be NP-hard to compute (see \cite{HendOlshec10}). In Section \ref{Afdim} we investigate this problem when $(\X,\Y) = (A, A)$ (see Example \ref{AA}) where $A$ will be a subalgebra of  $M_d^p(\C)$ for $d \in \Z_{\geq 1}$.

\begin{example}\label{INFcolrow}
    Let $p \in [1, \infty)$. As a generalization of 
    Example \ref{colrow}, 
    given any $L^p$-module $(\Y, \X)$ over an $L^p$-operator
    algebra $A \subseteq \Li(L^p(\mu_0))$, we now 
    define 
    \begin{align*}M^p_{1, \infty}(\Y) & =\ell^{p'}(\Z_{\geq 1}) \otimes_p \Y, \\
        M^p_{\infty, 1}(\X) & =\ell^p(\Z_{\geq 1})\otimes_p \X.
    \end{align*}
    Then $( M^p_{1,\infty}(\Y),M^p_{\infty, 1}(\X))$ is again an $L^p$-module over $A$, which is interpreted as a countable direct sum of the module $(\Y, \X)$. In fact, when $(\Y, \X)=(A, A)$, it follows from Theorem 3.13 and Proposition 3.15 in \cite{Delfin24lpmodules} that 
    $(M^p_{1,\infty}(A),M^p_{\infty, 1}(A))$ is a generalization of the standard Hilbert module $\Hi_A=\ell^2 \otimes A$ of a C*-algebra $A$. 
\end{example}

\section{$L^p$-modules over matrix algebras}\label{Afdim}
Let $p \in [1, \infty)$, let $d,n \in \Z_{\geq 1}$, and let $A$ be a subalgebra of 
$M_d^p(\C)$. Then $A$ is naturally an $L^p$-operator algebra acting on $\ell_d^p$, making $(A, A)$ an 
$L^p$-module over $A$ as in Example \ref{AA}. Further, 
if $A$ contains the unit of $M_d^p(\C)$ (that is 
the $d \times d$ identity matrix), then $(A, A)$ 
is a C*-like $L^p$-module.

For the rest of this section we consider the $L^p$-module $(M^p_{1,n}(A),M^p_{1,n}(A))$ for
$n \in \Z_{\geq 1}$. Notice that 
\[
M^p_{1,n}(A) \subseteq M^p_{n,1}(M_d^p(\C)) = M^p_{d,nd}(\C),
\]
and similarly  \[
M^p_{n,1}(A) \subseteq M^p_{n,1}(M_d^p(\C)) = M^p_{nd,d}(\C).
\]
In fact, the norms in $M^p_{n,1}(A)$ and $M^p_{1,n}(A)$ are the inherited 
operator norms from $M^p_{nd,d}(\C)$ 
and $M^p_{d,nd}(\C)$, respectively (see Notation \ref{MatrixCOnv}). Therefore, 
it now follows from Remark \ref{DefinitionLpModRmks} and the fact that we are working in only finite dimensional vector spaces that 
the module $(M^p_{1,n}(A),M^p_{n,1}(A))$ will 
be C*-like if we prove that 
for any $\vf{a} \in M^p_{n,1}(A)$
\begin{equation}\label{leqA}
\| \vf{a} \|_{M_{nd,d}^p(\C)} \leq 
\max \left\{
\| (\vf{b} \mid \vf{a})_A \|
\colon 
\vf{b}\in \B\big( {M^p_{1,n}(A)} \big)
\right\}, 
\end{equation}
and that for all $\vf{b} \in M^p_{1,n}(A)$
\begin{equation}\label{leqB}
\| \vf{b} \|_{M_{d,nd}^p(\C)} \leq  \max \left\{ 
\| (\vf{b} \mid \vf{a})_A \|
\colon 
 \vf{a}\in \B\big( M^p_{n,1}(A) \big)
\right\}.
\end{equation}
The following results provide 
instances of $A \subseteq M_d^p(\C)$ for which $(M^p_{1,n}(A),M^p_{n,1}(A))$ is C*-like. 

We will start with the case in which $A$ is $M_d^p(\C)$. 

\begin{theorem}\label{MainT1}
    Let $p \in [1, \infty)$, let $n,d \in \Z_{\geq 1}$, and let $A=M_{d}^p(\C)$. 
    Then $(M^p_{1,n}(A),M^p_{n,1}(A))$ is a C*-like $L^p$-module over $A$.
\end{theorem}
\begin{proof}
We first prove that the inequality in Equation \eqref{leqA} holds. To do so, 
for each $\eta \in \ell_{nd}^{p'}$ we define $\vf{b}_\eta \in M^p_{1,n}(M_d^p(\C))$ by letting 
\begin{equation}\label{b_etas}
   \vf{b}_\eta = \begin{bmatrix}
    \eta^T \\
    {\bf{0}}_{(d-1), nd}
\end{bmatrix}, 
\end{equation}
where, for $j,k \in \Z_{\geq 1}$ the notation ${\bf{0}}_{j,k}$ represents the $j \times k$ matrix with zero entries. 
Hence, for any fixed $\vf{a} \in M^p_{n,1}(M_d^p(\C))$ and $\xi \in \ell_d^p$, 
we have 
\begin{equation}\label{b_etaaxi}
(\vf{b}_\eta \mid \vf{a})_A\xi = 
\begin{bmatrix}
    \langle \eta, \vf{a}\xi \rangle_p \\
    {\vf{0}}_{d-1, 1}
\end{bmatrix} .
\end{equation}
Moreover, using Theorem \ref{HoldDual} at the last step gives 
\[
\| \vf{b}_\eta \|_{M_{d,nd}^p(\C)} = \max\{\|\vf{b}_\eta \zeta \| \colon  \zeta \in \B(\ell_{nd}^p)\}= \max \{ |\langle \eta, \zeta\rangle_p| \colon \zeta \in \B(\ell_{nd}^p) \} = \|\eta\|_{p'}.
\]
This shows that $\{b_\eta \colon \eta \in \B\big(\ell_{nd}^{p'}\big)\} \subseteq \B(M^p_{d,nd}(\C))$ and therefore 
\[
\max \{ 
\| (\vf{b} \mid \vf{a})_A \|
\colon  \vf{b} \in \B( M_{d,nd}^p(\C))  
\}
 \geq \max\{ \| (\vf{b}_\eta \mid \vf{a})_A \|
\colon \eta \in \B(\ell_{nd}^{p'}) \}.
\]
Thus, using \eqref{b_etaaxi} at the second 
step and
Corollary \ref{FiniteDopNorm} at the last step, 
we now get
\begin{align*}
\max\{ \| (\vf{b}_\eta \mid \vf{a})_A \|
\colon \eta \in \B\big(\ell_{nd}^{p'}\big) \} 
& = 
 \max\{ \| (\vf{b}_\eta \mid \vf{a})_{A}\xi \|_{p}
\colon \eta \in \B\big(\ell_{nd}^{p'}\big), \xi \in \B(\ell_d^p) \} \\
& = 
 \max\{ | \langle \eta, \vf{a}\xi \rangle_p|
\colon \eta \in \B\big(\ell_{nd}^{p'}\big), \xi \in \B(\ell_d^p) \} \\
& = \| \vf{a} \|_{M_{nd,d}^p(\C)}.
\end{align*}
This proves the desired inequality from Equation \eqref{leqA}. 

Next, we show that \eqref{leqB} also holds. The idea is similar, but this time we do not require Hölder duality. 
For any $\zeta \in \ell_{nd}^p$ we now 
define $\vf{a}_\zeta \in M^p_{n,1}(M_d^p(\C))$ 
as
\begin{equation}\label{a_zetas}
  \vf{a}_\zeta = \begin{bmatrix}
    \zeta & 
    \vf{0}_{nd, (d-1)}
\end{bmatrix}.  
\end{equation}
A direct computation shows that for any $\xi \in \ell_{d}^p$ we get $\vf{a}_\zeta \xi = \xi(1) \zeta$, 
whence for any $\vf{b} \in M^p_{1,n}(M_d^p(\C))$ 
we get
\begin{equation}\label{bazetaxi}
    (\vf{b} \mid \vf{a}_\zeta)_A\xi = \xi(1)\vf{b}\zeta.
\end{equation}
In particular, observe that 
\[
\|\vf{a}_\zeta\|_{M_{nd,d}^p(\C)}=\max_{\xi \in \B(\ell_d^p)} \|\vf{a}_\zeta\xi\|_p=
\max_{\xi \in \B(\ell_d^p)} |\xi(1)|\| \zeta\|_p=\|\zeta\|_p,
\]
and therefore 
\[
\max\{\| (\vf{b} \mid \vf{a})_A \|
\colon  \vf{a} \in \B( M_{nd,d}^p(\C))  \}
\geq  \max\{    \|(\vf{b} \mid \vf{a}_\zeta)_A\|\colon \zeta \in \B(\ell_{nd}^p)\}.
\]
Finally, using Equation \eqref{bazetaxi} at the second step, 
we conclude 
\begin{align*}
\max\{    \|(\vf{b} \mid \vf{a}_\zeta)_A\| \colon \zeta \in \B(\ell_{nd}^p)\} & = \max\{    \|(\vf{b} \mid \vf{a}_\zeta)_A\xi\|_p \colon \zeta \in \B(\ell_{nd}^p), \xi \in \B(\ell_d^p)\} \\
& = \max\{    |\xi(1)|\|\vf{b}\zeta\|_p \colon \zeta \in \B(\ell_{nd}^p), \xi \in \B(\ell_d^p)\}\\
& = \| \vf{b}\|_{M_{d,nd}^p(\C)}.
\end{align*}
The inequality in Equation \eqref{leqB} now follows and we are done. 
\end{proof}

    Observe that the proof of Theorem \ref{MainT1} relies 
    on the fact that the matrices $\vf{b}_\eta$ and 
    $\vf{a}_\zeta$ are actually elements of $M^p_{1,n}(A)$ 
    and $M^p_{n,1}(A)$ respectively. This in turn is 
    possible because, in this case, $A$ consists of all 
    matrices in $M_{d}^p(\C)$. Therefore, for a non-trivial 
    subalgebra of $M_{d}^p(\C)$ we can no longer 
    guarantee that $\vf{b}_\eta \in M^p_{1,n}(A)$ and 
    $\vf{a}_\zeta \in M^p_{n,1}(A)$. 
In Theorem \ref{MainT2} we will show that the C*-likeness of the 
    module $(M^p_{1,n}(A),M^p_{n,1}(A))$ remains valid for 
    any subalgebra $A$ of $M_{d}^p(\C)$ consisting of 
    block diagonal matrices. Particular cases of this more general setting include of 
    course $M_{d}^p(\C)$ itself and the algebra of diagonal matrices in  
    $M_{d}^p(\C)$. Thus, Theorem \ref{MainT1} can be seen 
    as a particular instance of Theorem \ref{MainT2} below. 
    However, to prove the general case we will need some pieces already proved in  Theorem \ref{MainT1}, which we record below in Corollary \ref{middleSteps}. The main reason of breaking it into two separate results is that it makes the exposition and arguments more clear. 

    \begin{corollary}\label{middleSteps}
            Let $p \in [1, \infty)$, let $n,d \in \Z_{\geq 1}$, and let $A=M_{d}^p(\C)$. For each $\vf{a} \in M_{nd,d}^p(\C)$ there is $\eta \in \B(\ell^{p'}_{nd})$ such that 
        \[
        \|\vf{a}\|_{M_{nd,d}^p(\C)} = \| (\vf{b}_\eta \mid \vf{a})_A\|,
        \]
        where $\vf{b}_\eta \in M^p_{1,n}(A)=M_{d,nd}^p(\C)$ is as in Equation \eqref{b_etas}. Similarly, for each $\vf{b} \in M_{d,nd}^p(\C)$ there is $\zeta \in \B(\ell^{p}_{nd})$ such that 
        \[
        \|\vf{b}\|_{M_{d,nd}^p(\C)} = \| (\vf{b} \mid \vf{a}_\zeta)_A\|,
        \]
         where $\vf{a}_\zeta \in M^p_{n,1}(A)=M_{nd,d}^p(\C)$ is as in Equation \eqref{a_zetas}.
    \end{corollary}
\begin{proof}
    In the proof of Theorem \ref{MainT1} it was shown that 
    \[
    \max\{ \| (\vf{b}_\eta \mid \vf{a})_A \|
\colon \eta \in \B\big(\ell_{nd}^{p'}\big) \} 
= \| \vf{a} \|_{M_{nd,d}^p(\C)}.
    \]
    and that 
    \[
    \max\{    \|(\vf{b} \mid \vf{a}_\zeta)_A\| \colon \zeta \in \B(\ell_{nd}^p)\} = \|\vf{b}\|_{M_{d,nd}^p(\C)}.
    \]
The desired result now follows by observing that $\eta \mapsto  \| (\vf{b}_\eta \mid \vf{a})_A\|$ is a continuous function $\ell^{p'}_{nd} \to \R$,  $\zeta \mapsto \|(\vf{b} \mid \vf{a}_\zeta)_A\|$ is a continuous function $\ell_{nd}^p \to \R$, and that both $\B(\ell^{p'}_{nd})$ and $\B(\ell^{p}_{nd})$ are compact due to finite dimension. 
\end{proof}

    We now begin by establishing some notation 
    and auxiliary results to prepare for Theorem \ref{MainT2}. 

    \begin{notation}
    Let $k \in \Z_{\geq 1}$, let $d_1, \ldots, d_k \in \Z_{\geq 1}$, for each $j \in \{1, \ldots, k\}$
    let $a_j$ be an $d_j \times d_j$ matrix, and let $d=d_1 + \cdots + d_k$. 
    Then $\op{diag}(a_1, \ldots, a_k)$ represents 
    the $d \times d$ block diagonal matrix that has the square matrices $a_1, \ldots, a_k$ along the diagonals and all the other entries equal to $0$. 
    \end{notation}

    \begin{definition}[Block-diagonal subalgebras]
       Let $p \in [1,\infty)$. For $d \in \Z_{\geq 1}$, we let $c(d,k)$ be a fixed composition of $d$ of length $k \in \Z_{\geq 1}\cap \Z_{\leq d}$. 
That is, $c(d,k)=(d_1, \ldots, d_k) \in \Z_{\geq 1}^k$ such that $d_1 + \cdots + d_k = d$. The \emph{block-diagonal} subalgebra of $M^p_d(\C)$ associated with the composition $c(d,k)$ is defined as 
\[
A_{c(d,k)} = 
\{\op{diag}(a_1, \ldots, a_k) \colon a_j \in M^p_{d_j}(\C)\}
\]
    \end{definition}
Observe that there is only one composition of length 
    $1$, namely $(d)$,
    and also there is only one composition of length $d$, 
    which is $(1,\ldots, 1) \in \Z_{\geq 1}^d$. 
    Furthermore, $A_{(d)}=M^p_d(\C)$ and $A_{(1,\ldots, 1)} \subseteq M^p_d(\C)$ is 
    the  subalgebra of diagonal matrices. 

    \begin{notation}\label{Block_Notation}   
        Let $p \in [1,\infty)$, let $d, n \in \Z_{\geq 1}$ and let $c(d,k)$ be a composition of $d$ of length $k$. If $\vf{a} \in M^p_{n,1}(A_{c(d,k)}) \subseteq M_{nd,d}^p(\C)$ and $j \in \{1,\ldots, k\}$, then we denote by 
        $\vf{a}(\colon, [j]) \in M^p_{nd,d_j}(\C)$ the submatrix of $\vf{a}$ 
        that only considers the \textit{columns} corresponding to the $j$-th block. For example, let $\vf{a} \in M^p_{2,1}(A_{(1,2,1)}) \subset M_{8,4}^p(\C)$ be given by 
        \begin{equation}\label{workingEx}
        \vf{a} = 
        \begin{bmatrix}
            a_1 & 0 & 0 & 0 \\
            0 & a_{11} & a_{12}& 0 \\
            0 & a_{21} & a_{22} & 0 \\
            0 & 0 & 0 & a_2\\
                 b_1 & 0 & 0 & 0 \\
            0 & b_{11} & b_{12}& 0 \\
            0 & b_{21} & b_{22} & 0 \\
            0 & 0 & 0 & b_2
        \end{bmatrix}.
        \end{equation}
        Then 
        \[
         \vf{a}(\colon, [1]) = 
                 \begin{bmatrix}
            a_1\\
            0  \\
            0  \\
            0 \\
            b_1     \\
            0  \\
            0 \\
            0 
        \end{bmatrix}, 
        \vf{a}(\colon, [2]) = 
                \begin{bmatrix}
       0 & 0  \\
       a_{11} & a_{12} \\
       a_{21} & a_{22}  \\
       0 & 0 \\
       0 & 0  \\
          b_{11} & b_{12}\\
         b_{21} & b_{22}  \\
        0 & 0 
        \end{bmatrix}, 
        \text{ and } 
          \vf{a}(\colon, [3]) =          \begin{bmatrix}
         0 \\
         0 \\
         0 \\
        a_2\\
         0 \\
          0 \\
          0 \\
           b_2
        \end{bmatrix}.
        \]
        Next, we will denote by $\vf{a}_{[j]} \in M^p_{nd_j,d_j}(\C)$ the submatrix of $\vf{a}(\colon, [j])$ that only extracts the $n$ building $j$-th blocks. For instance, using again the matrix from 
        Equation \eqref{workingEx} we have 
        \[
        \vf{a}_{[1]} = 
        \begin{bmatrix}
            a_1\\
            b_1
        \end{bmatrix},
         \vf{a}_{[2]} = 
                         \begin{bmatrix}
       a_{11} & a_{12} \\
       a_{21} & a_{22}  \\
          b_{11} & b_{12}\\
         b_{21} & b_{22}  
        \end{bmatrix}, 
        \text{ and } 
          \vf{a}_{[3]} =          \begin{bmatrix}
        a_2\\
           b_2
        \end{bmatrix}.
        \]
        Similarly, for any $\vf{b}\in M^p_{1,n}(A_{c(d,k)}) \subseteq M_{nd,d}^p(\C)$  and $j \in \{1,\ldots, k\}$, we denote by 
        $\vf{b}([j], \colon) \in M^p_{d_j,nd}(\C)$ to the submatrix of $\vf{b}$ 
        that only considers the \textit{rows} corresponding to the $j$-th block and denote by $\vf{b}_{[j]} \in M_{d_j,nd_j}^p(\C)$ to the submatrix of $\vf{b}([j], \colon)$ consisting of the $n$ building $j$-th blocks. 
    \end{notation}

Adopting the notation just introduced, for any $j \in \{1, \ldots, k\}$, we have 
\[
\| \vf{a}(\colon, [j]) \|_{M^p_{nd,d_j}(\C)} = \|\vf{a}_{[j]}\|_{M^p_{nd_j,d_j}(\C)} \ \text{ and } \  
\| \vf{b}([j], \colon) \|_{M^p_{d_j, nd}(\C)} = \|\vf{b}_{[j]}\|_{M^p_{d_j, nd_j}(\C)}.
\]
These norm equalities are straightforward to verify 
and will be used in the proof of Theorem \ref{MainT2}.
In fact, in the next lemma we show how these norms are used to compute the norm of the 
larger matrices $\vf{a} \in M_{n,1}(A_{c(d,k)})$ and $\vf{b} \in M_{1,n}(A_{c(d,k)})$. This, besides being an interesting result in its own right, will also be used repeatedly in the proof of Theorem \ref{MainT2}. 
    \begin{lemma}\label{KeyBlockNorms}
 Let $p \in [1, \infty)$, let $d \in \Z_{\geq 1}$, 
     let $n \in \Z_{\geq 1}$, and let $c(d,k)$ be a fixed composition of $d$.  For any $\vf{a} \in M^p_{n,1}(A_{c(d,k)})$ we have 
        \[
\| \vf{a} \|_{M_{nd,d}^p(\C)} = \max \left\{ \|\vf{a}(\colon, [j])\|_{M_{nd,d_j}^p(\C)} \colon j \in \{1, \ldots, k\}\right\}.
        \]
        Similarly, for $\vf{b} \in M^p_{1,n}(A_{c(d,k)})$ we have 
        \[
\| \vf{b} \|_{M_{d,nd}^p(\C)} = \max\left\{ \| \vf{b}([j], \colon ) \|_{M_{d_j, nd}^{p}(\C)} \colon j \in \{1, \ldots k\}\right\}. 
        \]
    \end{lemma}
    \begin{proof}
        Fix $\vf{a} \in M^p_{n,1}(A_{c(d,k)})$. For any 
        $\xi \in \B(\ell_d^p)$, we write $\xi=(\xi_1, \ldots, \xi_k)$ where each $\xi_j \in \ell_{d_j}^p$ and $\| \xi\|^p_p=\sum_{j=1}^k\|\xi_j\|_p^p \leq 1$. Then,
        \[
\| \vf{a} \xi \|_p^p = \sum_{j=1}^k \| \vf{a}(\colon, [j])\xi_j\|_p^p \leq \max \left\{ \|\vf{a}(\colon, [j])\|^p_{M_{nd,d_j}^p(\C)} \colon j \in \{1, \ldots, k\}\right\}.
        \]
        This proves $\| \vf{a} \|_{M_{nd,d}^p(\C)} \leq \max \left\{ \|\vf{a}(\colon, [j])\|_{M_{nd,d_j}^p(\C)} \colon j \in \{1, \ldots, k\}\right\}$. For the reverse inequality, for each $j \in \{1, \ldots, k\}$ let $\iota_j \colon \ell_{d_j}^p \to \ell_{d}^p$ be the natural isometric inclusion with respect to the fixed composition. Then, for $j \in \{1, \ldots, k\}$ and any $\zeta \in \B(\ell_{d_j}^p)$, we get 
        \[
\| \vf{a} \|_{M_{nd,d}^p(\C)} \geq \| \vf{a}\iota_j(\zeta)\|_p = \| a(\colon, [j])\zeta\|_p,
        \]
        from where it follows that $\| \vf{a} \|_{M_{nd,d}^p(\C)} \geq \| \vf{a}(\colon, [j])\|_{M_{nd,d_j}^p(\C)}$ for each 
        $j \in \{1, \ldots, k\}$, proving the desired reverse inequality. 

        For the second part of the statement, notice first that the result is immediate when $p=1$, for $ \| \vf{b} \|_{M_{d,nd}^1(\C)}$ is the maximum $1$-norm of the columns of $\vf{b}$ and similarly each $\|\vf{b}([j], \colon)\|_{M_{d_j, nd}^{1}(\C)}$ is the maximum $1$-norm of the columns of $\vf{b}([j],\colon)$, whence
        \[
        \| \vf{b} \|_{M_{d,nd}^1(\C)} = \max \left\{ \|\vf{b}([j], \colon)\|_{M_{d_j, nd}^{1}(\C)} \colon j \in \{1, \ldots, k\}\right\}.
        \]
        For $p \in (1,\infty)$, first observe that 
$\vf{b} \in M^p_{1,n}(A_{c(d,k)})$ implies $\vf{b}^T \in M_{n,1}^{p'}(A_{c(d,k)})$. Then an application of Corollary \ref{a_p=a^T_q}, followed by the first part of this proposition for $p' \in (1, \infty)$, and finally a second application of Corollary \ref{a_p=a^T_q}, yields
        \begin{align*}
        \| \vf{b} \|_{M_{d,nd}^p(\C)} = \| \vf{b}^{T} \|_{M_{nd,d}^{p'}(\C)} & = \max \left\{ \|\vf{b}^T(\colon, [j])\|_{M_{nd,d_j}^{p'}(\C)} \colon j \in \{1, \ldots, k\}\right\}\\
         & = \max \left\{ \|\vf{b}([j], \colon)\|_{M_{d_j, nd}^{p}(\C)} \colon j \in \{1, \ldots, k\}\right\},
        \end{align*}
        finishing the proof. 
    \end{proof}

    \begin{theorem}\label{MainT2}
    Let $p \in [1, \infty)$, let $n,d \in \Z_{\geq 1}$, and let $c(d,k)$ be a fixed composition of $d$.
    Then $(M^p_{1,n}(A_{c(d,k)}),M^p_{n,1}(A_{c(d,k)}))$ is a C*-like $L^p$-module over $A_{c(d,k)}$.
\end{theorem}
\begin{proof}
    As before, we only need to establish that the inequalities in Equations \eqref{leqA} and \eqref{leqB} hold for $A=A_{c(d,k)}$. 

    First, take $\vf{a} \in M_{n,1}(A_{c(d,k)})$ and for each $j \in \{1, \ldots, k\}$ let $\vf{a}_{[j]} \in M_{nd_j,d_j}^p(\C)$ be as defined in Notation \ref{Block_Notation}. Corollary \ref{middleSteps} gives $\eta_j \in \B(\ell_{nd_j}^{p'})$ such that 
    \begin{equation}\label{vert_block_norms}
    \| \vf{a}(\colon, [j])\|_{M_{nd,d_j}^p(\C)}= \| \vf{a}_{[j]}\|_{M_{nd_j,d_j}^p(\C)}=\|(\vf{b}_{\eta_j} \mid \vf{a}_{[j]})_{M^p_{d_j}(\C)}\|,
    \end{equation}
    where $\vf{b}_{\eta_j} \in M_{d_j,nd_j}^p(\C)$ is as defined in Equation \eqref{b_etas}. Next, for each $j \in \{1, \ldots, k\}$, we decompose $\eta_j$ into its $n$ $j$-th block parts: 
    \[
    \eta_j = (\eta_{j1}, \eta_{j2}, \ldots, \eta_{jn}).
    \]
    That is, each $\eta_{jl} \in \ell_{d_j}^{p'}$ and
    \[
    \eta_j^T = \begin{bmatrix}
        \eta_{j1}^T & \eta_{j2}^T & \ldots & \eta_{jn}^T
    \end{bmatrix}.
    \]
     Therefore $\vf{b}_{\eta_j}$ gets the following square blocks decomposition:
    \[
    \vf{b}_{\eta_j} = 
    \begin{bmatrix}
        \vf{b}_{\eta_{j1}} &  \vf{b}_{\eta_{j2}} & \ldots &  \vf{b}_{\eta_{jn}}
    \end{bmatrix},
    \]
    where $\vf{b}_{\eta_{jl}} \in M_{d_j}^p(\C)$ for any $l \in \{1, \ldots, n\}$.
Now, define 
    \[
    \vf{b}_0=
    \begin{bmatrix}
        \op{diag}(\vf{b}_{\eta_{11}},\ldots, \vf{b}_{\eta_{k1}})
        & 
        \op{diag}(\vf{b}_{\eta_{12}},\ldots, \vf{b}_{\eta_{k2}}) & \ldots &
         \op{diag}(\vf{b}_{\eta_{1n}},\ldots, \vf{b}_{\eta_{kn}})
    \end{bmatrix}.
    \]
    By construction we have $\vf{b}_0\in M_{1,n}^p(A_{c(d,k)})$
    and ${(\vf{b}_0)}_{[j]} = \vf{b}_{\eta_j}$, whence
    \[
    \|\vf{b}_0([j], \colon) \|_{M_{d_j,nd}^p(\C)}= \|{(\vf{b}_0)}_{[j]}\|_{M_{d_j,nd_j}^p(\C)}=\| \vf{b}_{\eta_j} \|_{M_{d_j,nd_j}^p(\C)} = \| \eta_j \|_{p'}.
    \]
    Thus, using Lemma \ref{KeyBlockNorms}, we get 
\[
\| \vf{b}_0 \|_{M_{d,nd}^p(\C)} = \max\left\{ \| \eta_j \|_{p'} \colon j \in \{1, \ldots k\}\right\}  \leq 1,
\]
that is $\vf{b}_0 \in \B(M_{1,n}^p(A_{c(d,k)}))$.
Furthermore, a direct computation using the block structure gives
\[
( \vf{b}_0 \mid \vf{a} )_{A_{c(d,k)}} = 
\op{diag}( 
(\vf{b}_{\eta_1} \mid \vf{a}_{[1]})_{M^p_{d_1}(\C)}, 
\ldots, 
(\vf{b}_{\eta_k} \mid \vf{a}_{[k]})_{M^p_{d_k}(\C)}
).
\]
Two applications of Lemma \ref{KeyBlockNorms} 
together with Equation \eqref{vert_block_norms}
now give
\begin{align*}
\| ( \vf{b}_0 \mid \vf{a} )_{A_{c(d,k)}} \| & = 
\max\{ \|(\vf{b}_{\eta_j} \mid \vf{a}_{[j]})_{M^p_{d_j}(\C)}\| \colon j \in \{1,\ldots, k\}\} \\
& =    \max\{ \| \vf{a}(\colon, [j])\|_{M_{nd,d_j}^p(\C)}   \colon j \in \{1,\ldots, k\}\} \\
& = \|\vf{a}\|_{M_{nd,d}^p(\C)},
\end{align*}
proving \eqref{leqA} when $A=A_{c(d,k)}$. 

Next, fix $\vf{b} \in M_{1,n}(A_{c(d,k)})$ and, for 
each $j \in \{1, \ldots, k\}$ we use Corollary \ref{middleSteps} 
to find $\zeta_j \in \B(\ell^p_{nd_j})$ such that 
\begin{equation}\label{hor_block_norms}
\|\vf{b}([j], \colon)\|_{M_{d_j,nd}^p(\C)} =\| \vf{b}_{[j]} \|_{M_{d_j,nd_j}^p(\C)}=\| (\vf{b}_{[j]} \mid \vf{a}_{\zeta_j})_{M_{d_j}^p(\C)}\|,
\end{equation}
where $\vf{a}_{\zeta_j} \in M_{nd_j, d_j}^p(\C)$ is as defined in Equation $\eqref{a_zetas}$. As before, we let $\zeta_j = (\zeta_{j1}, \zeta_{j2}, \ldots, \zeta_{jn})$
be the decomposition of $\zeta_j$ into its $n$ $j$-th block parts, so that each $\zeta \in \ell_{d_j}^p$ and
\[
\vf{a}_{\zeta_j} = 
\begin{bmatrix}
    \vf{a}_{\zeta_{j1}}\\
    \vdots\\
    \vf{a}_{\zeta_{jn}}
\end{bmatrix}
\]
with $a_{\zeta_{j,l}} \in M_{d_j}^p(\C)$ for any $l \in \{1, \ldots, n\}$. Now we define 
\[
\vf{a}_0 = 
\begin{bmatrix}
    \op{diag}(\vf{a}_{\zeta_{11}},  \ldots, \vf{a}_{\zeta_{k1}}) \\
    \op{diag}(\vf{a}_{\zeta_{12}}, \ldots, \vf{a}_{\zeta_{k2}}) \\
    \vdots\\
    \op{diag}(\vf{a}_{\zeta_{1n}},  \ldots, \vf{a}_{\zeta_{kn}}) 
\end{bmatrix}\in M_{n,1}(A_{c(d,k)}).
\]
Moreover, notice that $(\vf{a}_0)_{[j]}=\vf{a}_{\zeta_j}$, whence
\[
\|\vf{a}_{0}(\colon, [j]) \|_{M_{nd,d_j}^p} = \|(\vf{a}_0)_{[j]}\|_{M_{nd_j,d_j}^p(\C)} = \| \vf{a}_{\zeta_j}\|_{M_{nd_j,d_j}^p(\C)} = \| \zeta_j\|_p.
\]
Lemma \ref{KeyBlockNorms} now gives
\[
\|\vf{a}_0 \|_{M_{nd,d}^p(\C)}=\op{max}\{ \| \zeta_j\|_p \colon j \in \{1,\ldots, k\}\} \leq 1,
\]
that is, $\vf{a}_0 \in \B(M_{n,1}(A_{c(d,k)}))$. Finally, 
a direct computation gives 
\[
( \vf{b} \mid \vf{a}_{0})_{A_{c(d,k)}} 
= \op{diag}\big(
(\vf{b}_{[1]} \mid \vf{a}_{\zeta_1})_{M_{d_1}^p(\C)}, 
\ldots, 
 (\vf{b}_{[k]} \mid \vf{a}_{\zeta_k}
 )_{M_{d_k}^p(\C)}
  \big),
\]
so two final applications of Lemma \ref{KeyBlockNorms}
together with Equation \eqref{hor_block_norms} now give 
\[
\| ( \vf{b} \mid \vf{a}_{0})_{A_{c(d,k)}} \| = \| \vf{b}\|_{M_{d,nd}^p(\C)},
\]
from where we conclude that \eqref{leqB} holds when $A=A_{c(d,k)}$, finishing the proof.
\end{proof}

We end this section by observing that the same proofs we used in Theorems \ref{MainT1}
and \ref{MainT2} can be easily adapted to the standard $p$-module case from Example \ref{INFcolrow}. 

Indeed, let $d \in \Z_{\geq1}$ and $p \in [1,\infty)$. If $A \subseteq M_d^p(\C)$, then an element of $M^p_{\infty, 1}(A)$ is a sequence of matrices in $A$, indexed by $\Z_{\geq 1}$, such that when arranged as a `$\infty \times d$' matrix its $d$ columns are elements of $\ell^p(\Z_{\geq 1})$. Dually, an element of $M^p_{1,\infty}(A)$ is also a sequence of matrices in $A$, indexed by $\Z_{\geq 1}$, such that when arranged as a `$d \times \infty$' matrix its $d$ rows are elements of $\ell^{p'}(\Z_{\geq 1})$. Therefore, we can now state our most general result as follows:

    \begin{theorem}\label{MainTT}
    Let $p \in [1, \infty)$, let $n \in \Z_{\geq 1} \cup \{\infty\}$, let $d \in \Z_{\geq 1}$, and let $c(d,k)$ be a fixed composition of $d$.
    Then $(\ M^p_{1,n}(A_{c(d,k)}),M^p_{n,1}(A_{c(d,k)}) \ )$ is a C*-like $L^p$-module over $A_{c(d,k)}$.
\end{theorem}

\section{Counterexamples: non C*-like modules over $M_d^p(\C)$}\label{CounterExs}

It is clear that any algebra that satisfies the hypothesis of Theorem \ref{MainTT} is a unital subalgebra of $M_d^p(\C)$. It is easy to find examples of the C*-likeness  failing when we drop the unital assumption: 

\begin{example}\label{ex_UpperTriang}
    Let $p \in [1, \infty)$, and let $A=T^p_2 \subset M_2^p(\C)$ be the algebra of strictly upper triangular matrices. That is, 
    \[
    A=\left\{
    \begin{bmatrix}
        0 & \lambda \\
        0 & 0 
    \end{bmatrix} \colon \lambda \in \C
    \right\}.
    \]
    Then, for any $n \in \Z_{\geq 1} \cup\{\infty\}$, if $(\vf{b},\vf{a}) \in  M^p_{1,n}(A)\times M^p_{n,1}(A)$ we have $\vf{b}\vf{a}=0\in A$. Hence, the C*-likeness condition \eqref{C*likeness1} only holds when $\vf{a}$ is 0 and similarly \eqref{C*likeness2} is true only when $\vf{b}$ is zero. This shows that  $(M^p_{1,n}(A),M^p_{n,1}(A))$ is never a C*-like module over $A$. In fact the case $n=1$ was already considered in Example \ref{AA}.
\end{example}

Unfortunately, $A$ being unital is not a sufficient condition for C*-likeness. 
In the following example we use the well-known fact that the $1$-operator norm of any complex valued matrix is the maximum $1$-norm of its columns to give an example of a unital subalgebra of $M_2^1(\C)$ such that $(M^1_{1,2}(A),M^1_{2,1}(A))$ is not C*-like. 

\begin{example}\label{SD}
Let $u \in M_2^1(\C)$ be the normalized Hadamard matrix of dimension 2. That is, 
\[
u=\frac{1}{\sqrt{2}}\begin{bmatrix}
    1 & 1 \\
    1 & -1
\end{bmatrix}.
\]
Now, let $A \subset M^1_2(\C)$ be the algebra of simultaneously diagonalizable matrices 
with basis of eigenvectors given by the columns of $u$. That is, 
\[
A = \{u \ \op{diag}(\lambda_1, \lambda_2) \ u^{-1} \colon \lambda_1, \lambda_2 \in \C\}. 
\]
Observe that $A$ is indeed unital. Next, we define $a_1, a_2 \in A$ by 
\[
a_1=u \ \op{diag}(2, 1) \ u^{-1}, a_2=u \ \op{diag}(1, 2) \ u^{-1}. 
\]
Since 
\[
\vf{a} = \begin{bmatrix}
    a_1\\
    a_2
\end{bmatrix}
= 
\frac{1}{2}
\begin{bmatrix}
    3 & 1 \\
    1 & 3 \\
    3 & -1\\
    -1&3
\end{bmatrix} \in M^1_{2,1}(A),
\]
we see at once that $\|\vf{a}\|_{M_{4,2}^1(\C)}=4$. On the other hand, observe that any $\vf{b} \in \B(M^1_{1,2}(A))$ looks like 
\[
\vf{b} = \frac{1}{2}
\begin{bmatrix}
    \lambda_1 + \lambda_2 & \lambda_1-\lambda_2 & \lambda_3+\lambda_4 & \lambda_3-\lambda_4 \\
        \lambda_1 - \lambda_2 & \lambda_1+\lambda_2 & \lambda_3-\lambda_4 & \lambda_3+\lambda_4 \\
\end{bmatrix},
\]
where $\lambda_1, \lambda_2, \lambda_3, \lambda_4 \in \C$ satisfy 
\begin{align*}
     |\lambda_1 + \lambda_2| + |\lambda_1-\lambda_2 | & \leq 2,\\
          |\lambda_3 + \lambda_4| + |\lambda_3-\lambda_4 | & \leq 2.
\end{align*}
For any such $\vf{b}$, we have 
\[
\vf{b}\vf{a} = \frac{1}2
\begin{bmatrix}
    2\lambda_1 + \lambda_2 + \lambda_3 + 2\lambda_4 & 2\lambda_1 - \lambda_2 + \lambda_3 - 2\lambda_4 \\
    2\lambda_1 - \lambda_2 + \lambda_3 - 2\lambda_4 & 2\lambda_1 + \lambda_2 + \lambda_3 + 2\lambda_4
\end{bmatrix},
\]
whence 
\[
\|( \vf{b} \mid \vf{a} ) \|_A = \frac{1}2| 2\lambda_1 + \lambda_2 + \lambda_3 + 2\lambda_4 |+\frac{1}{2}|2\lambda_1 - \lambda_2 + \lambda_3 - 2\lambda_4|.
\]
We claim the following:
\begin{claim}
$\max\{\|( \vf{b} \mid \vf{a} ) \|_A  \colon \vf{b} \in \B(M^1_{1,2}(A))\}=\sqrt{10}$.
\end{claim}
\begin{claimproof}
By the discussion above, finding $\max\{\|( \vf{b} \mid \vf{a} ) \|_A  \colon \vf{b} \in \B(M^1_{1,2}(A))\}$ require us to solve 
the following optimization problem:
\[
\max\left\{ \frac{1}2| 2\lambda_1 + \lambda_2 + \lambda_3 + 2\lambda_4 |+\frac{1}{2}|2\lambda_1 - \lambda_2 + \lambda_3 - 2\lambda_4| \ 
\colon \
\begin{matrix}
    |\lambda_1 + \lambda_2| + |\lambda_1-\lambda_2 |  \leq 2\\
      |\lambda_3 + \lambda_4| + |\lambda_3-\lambda_4 |  \leq 2
\end{matrix}
\right\}.
\]
Multiplying through by $2$ and reparameterizing as
\begin{align*}
    \lambda_1&=w_1+w_2=r_1e^{i\theta_1}+r_2e^{i\theta_2},\\
    \lambda_2&=w_1-w_2=r_1e^{i\theta_1}-r_2e^{i\theta_2},\\
    \lambda_3&=w_3+w_4=r_3e^{i\theta_3}+r_4e^{i\theta_4},\\
    \lambda_4&=w_3-w_4=r_3e^{i\theta_3}-r_4e^{i\theta_4},
\end{align*}
where $w_j\in\C,\theta_j\in\R, r_j\in\R_+$ for all $j \in \{1, \ldots, 4\}$,
the claim becomes equivalent to proving 
\[
    \max_{\substack{r_1+r_2\leq 1 \\ r_3+r_4\leq 1}}|3r_1e^{i\theta_1}+r_2e^{i\theta_2}+3r_3e^{i\theta_3}-r_4e^{i\theta_4}|+
|r_1e^{i\theta_1}+3r_2e^{i\theta_2}-r_3e^{i\theta_3}+3r_4e^{i\theta_4}|  = 2\sqrt{10}.
\]
The advantage now is that the objective function is convex in $r_1,r_2,r_3,r_4$ for fixed $\theta_1,\theta_2,\theta_3,\theta_4$.
Thus, it suffices to check only the extreme points of the polytope bounding the $r_1, r_2,r_3,r_4$:
\begin{equation*}
    ((r_1,r_2),(r_3,r_4))\in\{(0,0),(1,0),(0,1)\}^2.
\end{equation*}
By symmetry, there are only four distinct cases to check for the values of $(r_1,r_2,r_3,r_4)$:
\begin{enumerate}
\item $(0,0,0,0)$: The objective value is $0$.
\item $(0,0,1,0)$: The objective value is $4$.
\item $(1,0,1,0)$: Without loss of generality, we can rotate so that $\theta_1=0$, i.e. $w_1$ is real. The problem reduces to
    \begin{equation*}
        \max_\theta |3+3e^{i\theta}|+|1-e^{i\theta}|=\max_\theta \sqrt{18+18\cos\theta}+\sqrt{2-2\cos\theta}= 2\sqrt{10},
    \end{equation*}
maximized at $\theta=\arccos(4/5)$.
\item $(1,0,0,1)$: Again, rotate so that $\theta_1=0$. The problem reduces to
    \begin{equation*}
        \max_\theta |3-e^{i\theta}|+|1+e^{i\theta}|=\max_\theta \sqrt{10-6\cos\theta}+\sqrt{10+6\cos\theta} =2\sqrt{10},
    \end{equation*}
maximized at $\theta=\frac{\pi}{2}$.
\end{enumerate}
This proves the claim.
\end{claimproof}

Hence, using the claim we get  at once 
\[
\max\{\|( \vf{b} \mid \vf{a} ) \|_A  \colon \vf{b} \in \B(M^1_{1,2}(A))\}= \sqrt{10} < 4 = \|\vf{a}\|_{M_{4,2}^1(\C)}.
\]
Therefore, C*-likeness condition \eqref{C*likeness1} does not hold, showing that $(M_{1,2}^1(A),M_{2,1}^1(A))$ is not a C*-like module. 
\end{example}

We do not know yet whether Example \ref{SD}  works for $p \in (1,\infty)\setminus \{2\}$. Of course, for $p=2$, $A$ is a C*-algebra and therefore the example does not work (see Lemma \ref{C*innerRec}). We are currently working with both numerical computations and interpolation techniques to investigate Example \ref{SD} when $p \in (1,\infty)\setminus \{2\}$. 

\subsection*{Funding} 

All authors where partially supported by the Department of Mathematics at the University of Colorado, Boulder via the Summer 2024 REU(G) program.

\bibliographystyle{plain}
\bibliography{cstarpnorms.bib}
\end{document}